\documentclass[a4paper,12pt]{amsart}

\usepackage{a4wide}
\usepackage{amssymb}
\usepackage{amsthm}
\usepackage{amsmath}
\usepackage{amscd}
\usepackage{amsrefs}
\usepackage{verbatim}
\usepackage{mathrsfs}
\usepackage[english]{babel}
\usepackage[latin1]{inputenc}
\usepackage{latexsym}
\usepackage{float}
\usepackage{microtype}
\usepackage{hyperref}

\newcommand{\N}{\mathbb{N}}
\newcommand{\Z}{\mathbb{Z}}
\newcommand{\Q}{\mathbb{Q}}
\newcommand{\R}{\mathbb{R}}

\newcommand{\C}{\mathbb{C}}

\theoremstyle{plain}
\newtheorem{thm}{Theorem}[section]
\newtheorem{cor}[thm]{Corollary}
\newtheorem{lem}[thm]{Lemma}
\newtheorem{prop}[thm]{Proposition}
\newtheorem{con}[thm]{Conjecture}

\theoremstyle{definition}
\newtheorem{defi}[thm]{Definition}
\newtheorem{exmp}[thm]{Example}
\newtheorem{quest}[thm]{Question}
\newtheorem{rem}[thm]{Remark}

\allowdisplaybreaks

\begin{document}

\title{Radial Limits of the Universal Mock Theta Function \large $g_3$}
\author{Min-Joo Jang and Steffen L\"obrich}
\address{Mathematical Institute, University of Cologne, Weyertal 86-90, 50931 Cologne, \linebreak Germany}
\email{min-joo.jang@uni-koeln.de}
\email{steffen.loebrich@uni-koeln.de}
\date{\today}

\begin{abstract}
Referring to Ramanujan's original definition of a mock theta function, Rhoades asked for explicit formulas for radial limits of the universal mock theta functions $g_2$ and $g_3$. Recently, Bringmann and Rolen found such formulas for specializations of $g_2$. Here we treat the case of $g_3$, generalizing radial limit results for the rank generating function of Folsom, Ono, and Rhoades. Furthermore, we find expressions for radial limits of fifth order mock theta functions different from those of Bajpai, Kimport, Liang, Ma, and Ricci. 
\end{abstract}

\maketitle

\section{Introduction}

In his last letter to Hardy, Ramanujan provided 17 examples for what he called a ``mock theta function''. These are given by $q$-hypergeometric series that have the same growth behavior as classical modular forms as one radially approaches roots of unity from inside the unit disc. However, he conjectured that one single modular form is not enough to cut out all the poles at roots of unity. More precisely, Ramanujan defined mock theta functions as follows (cf. \cite{orgdef}, p.63).

\begin{defi}[Ramanujan]\label{ramdef}
A mock theta function is a function $F(q)$, defined by a $q$-series which converges for $|q|<1$, satisfying the following conditions:
\begin{itemize}
	\item[(i)] There are infinitely many roots of unity $\zeta$ such that $F(q)$ grows exponentially as $q$ approaches $\zeta$ radially from inside the unit disc.
	\item[(ii)] For every root of unity $\zeta$, there exists a weakly holomorphic modular form $M_{\zeta}(q)$ and a rational number $\alpha$, such that $F(q)-q^{\alpha}M_{\zeta}(q)$ is bounded as $q$ approaches $\zeta$ radially from inside the unit disc.
	\item[(iii)] There does not exist a single weakly holomorphic modular form $M(q)$ and a rational number $\alpha$, such that for every root of unity $\zeta$, the difference $F(q)-q^{\alpha} M(q)$ is bounded as $q$ approaches $\zeta$ radially from inside the unit disc.
\end{itemize}
\end{defi}
Griffin, Ono, and Rolen \cite{gor} proved that all of Ramanujan's examples satisfy Definition \ref{ramdef}. Rhoades \cite{rh} asked if one can choose the set $\{M_{\zeta}(q)\}_{\zeta}$ to be finite and can also give explicit expressions for the radial limits in (ii). In particular, he asked this for certain specializations of the universal mock theta functions $g_2 (x,q)$ and $g_3 (x,q)$, since every mock theta function can be expressed as a linear combination of such specializations, up to a rational $q$-powers, and modular forms. The problem has been solved for special mock theta functions before: Folsom, Ono, and Rhoades gave a closed expression for the radial limits of Ramanujan's third order mock theta function $f$ (Theorem 1.1 of \cite{for}) as well as a partial answer for Rhoades' question for certain specializations of $g_3$ (Theorem 1.2 of \cite{for}, see also Remark \ref{forrem}). Moreover, in \cite{bklmr}, the radial limits of Ramanujan's fifth order mock theta functions were computed using bilateral series.\\

Bringmann and Rolen \cite{br} gave a positive answer to Rhoades' question for specializations of the even order universal mock theta function $g_2$, i.e., functions of the form $F(q)=g_2 (\zeta_b ^a q^A , q^B)$ for $a,b,A,B\in\N$. Linear combinations of such functions and modular forms comprise all of Ramanujan's mock theta functions up to rational $q$-powers. They showed that at most four modular forms (including the form identically equal to $0$) are needed for such $F$ to keep the radial limits bounded. For this they used a bilateral series identity for $g_2$ found by Mortenson \cite{mor} together with a careful analysis of zeros and poles of $g_2$ for case-by-case estimates. In this paper, we study the odd order universal mock theta function $g_3$, defined as
\begin{equation*}
g_3 (x,q) :=\sum_{n\geq 1}\frac{q^{n(n-1)}}{(x, q/x ; q)_{n}}
\end{equation*}
(see Section 3 for a definition of the $q$-Pochhammer symbols $(a_1,\dots,a_m ;q)_n$).
Throughout we write $\zeta_k ^h$ for the root of unity $e^{2\pi i \frac{h}{k}}$. In this case, Rhoades' question for $g_3$ can be stated as follows: 

\begin{quest}[Rhoades \cite{rh}, \S 3, 3.4.]\label{rq}
For every $a,b,A,B\in\N$ with $(a,b)=1$, can one find for every $k,h\in \N$ with $(k,h)=1$ a weakly holomorphic modular form $M_{k,h}(q)$ satisfying (ii) of Definition \ref{ramdef}, such that the set $\{M_{k,h} \}_{k,h}$ is finite, and explicitly compute the radial limits 
\begin{equation*}
Q_{a,b,A,B,h,k}:=\lim_{q\rightarrow \zeta_k ^k} \left(g_3(\zeta_b ^a q^A, q^B)-M_{k,h}(q)\right)?
\end{equation*}
\end{quest}

\begin{rem}
Note that while for every root of unity $\zeta_k ^h$ one can choose a lot of modular forms $M_{k,h}(q)$ matching condition (ii) of Definition \ref{ramdef}, Choi, Lim, and Rhoades \cite{clr} showed that the expressions $Q_{a,b,A,B,h,k}$ are uniquely determined. Moreover, they proved that for every $a,b,A,B\in \N$, the function mapping $\frac{h}{k}\in\Q$ to $Q_{a,b,A,B,h,k}$ is a quantum modular form of weight $\tfrac{1}{2}$ whose cocycles extend to a real analytic function on $\R$, except possibly at one point (for an introduction to quantum modular forms, see \cite{zag}). Bringmann and Rolen constructed such quantum modular forms in a different guise in \cite{br2}.
\end{rem}

A positive answer to Question \ref{rq} follows from the work of Bringmann and Rolen on $g_2$ \cite{br}, since $g_3$ can be expressed in terms of $g_2$ (cf.~\cite{gmi}, eq.~(6.1)). However, this relation is quite involved, so finding suitable modular forms and computing the radial limits in the $g_3$-case can be laborious with this method. We affirm Rhoades' question for most of the choices for the parameters $a,b,A,B,h,k$ directly by constructing the set $\{M_{k,h}\}_{k,h}$ and giving explicit expressions for the radial limits. For this we need an identity analogous to Mortenson's (\cite{mor}, Corollary 5.2.) as well as a relation between certain $g_3$-values by Kang (\cite{kang}, Theorem 1.3). In doing so, we generalize Theorem 1.2 of \cite{for} and find a different proof for it. We also obtain non-trivial identities for roots of unity as a side product. However, for a few special choices of the parameters, we could not find a direct answer.\\

The paper is organized as follows: In Section 2 we will state our results explicitly. For this we have to distinguish several cases depending on the behavior of $g_3$ at the root of unity in question. In Section 3 we recall some definitions and identities needed for the proofs, which will be given in Section 4.

\section*{Acknowledgements}
This paper will be part of the first author's PhD thesis and was supported by the Deutsche Forschungsgemeinschaft
(DFG) Grant No.~BR 4082/3-1, and partially written in the context of the Cologne Young Researchers in Number Theory Program 2015, supported by University of Cologne postdoc DFG Grant D-72133-G-403-15100101 funded under the Institutional Strategy of the University of Cologne within the German Excellence Initiative. First of all, we would like to thank Prof.~Dr.~Kathrin Bringmann and Dr.~Larry Rolen for suggesting the problem and giving us many helpful instructions and remarks. A good portion of this paper is based on their ideas. Furthermore, we are grateful to Prof.~Dr.~Wadim Zudilin for many valuable remarks and a profitable conversation. We would also like to thank Dr.~Michael Griffin for his advice and Dr.~Eric Mortenson, Michael Somos, and the referee for useful correspondence.

\section{Statement of Results}

\begin{rem}\label{modu}
Throughout this paper we slightly abuse the term ``modular form''. To be precise, we mean by a \emph{modular form} a weakly holomorphic modular form (with respect to the $\tau$-variable) of weight $1/2$ for some congruence subgroup and some multiplier, up to multiplication by a rational $q$-power.
\end{rem}

There are several methods to determine the modular forms $M_{k,h}(q)$ and compute the radial limits, depending on the behavior of $g_3$ at the chosen root of unity. The most straightforward case is when the sum defining $g_3$ converges absolutely in the radial limit. In this case, we can choose $M_{k,h}$ to be identically zero. If the sum does not converge, some summands can have poles arising from zeros in the denominator. Since there are no zeros in the numerator and the summands have increasing pole order, there is no way for the poles to cancel out (cf.~Section 3 in \cite{br}). If neither of these cases occur, further analysis is required.\\

First, we consider the case where a pole occurs in the summands of $g_3$. Here we need a Mortenson-type bilateral series identity for $g_3$. 
Let $a,b,A,B,k \in \N$ with $b,k>0$ and $(a,b)=1$. For $g_3(\zeta_b ^a \zeta_k ^{Ah} e^{-At},\zeta_k ^{Bh} e^{-Bt})$ to be well-defined for every $t>0$, we want to assume that $B\nmid A$ whenever $b=1$. Now we set $k':=\frac{k}{(k,B)}$, $B':=\frac{B}{(k,B)}$, and
\begin{equation*}
\mathcal{Q}_{a,b,A,B}:=\left\lbrace \frac{h}{k}\in\Q : b | k \text{ and } (k,B)\Big| \left(\frac{ak}{b}+Ah \right)  \right\rbrace.
\end{equation*} 
As in \cite{br}, Section 3, this is the set for which some summands of $g_3$ have poles. All limits in this paper are to be understood as radial limits from within the unit disc. To be precise, we will denote by $\lim_{q\rightarrow \zeta}F(q)$ the limit $\lim_{t\rightarrow 0^+}F(\zeta e^{-t})$ for $\zeta$ a root of unity and $F$ a function on the unit disc. The Jacobi triple product $j(x,q):=(x,q/x,q;q)_\infty$ is defined in Section 3.

\begin{thm}\label{pole}
If $\frac{h}{k}\in \mathcal{Q}_{a,b,A,B}$, then
\begin{multline*}
\lim_{q\rightarrow \zeta_h ^k} \left( g_3 (\zeta_b ^a q^A, q^B)-\frac{(q^B, q^B)_\infty ^2 j(q^{\frac{B}{2}},q^B)^2}{\zeta_b ^{a} q^{A}j(\zeta_b ^{a}q^{A+\frac{B}{2}}, q^B)^2 j(\zeta_b ^{a}q^{A}, q^B)} \right)\\
= -\zeta_b ^{-a} \zeta_k ^{-hA} + \zeta_b ^{-2a} \zeta_k ^{-2hA}\sum_{n= 1}^{k'} (\zeta_b ^{-a} \zeta_k ^{h(B-A)};\zeta_k ^{hB})_{n-1} (\zeta_b ^a \zeta_k ^{hA}; \zeta_k ^{hB})_{n} \zeta_k ^{hBn}.
\end{multline*}
Here in the expression $q^{\frac{B}{2}}$ the choice of the possibly implied square root does not matter.
\end{thm}

\begin{rem}
In fact, more is true: The right hand side, if one replaces $\zeta_k ^h$ by $\zeta_k ^h e^{-t}$, is not just the limit, but an asymptotic expansion of the left hand side as $t\rightarrow 0^+$. The same holds for the other theorems below. 
\end{rem}

\begin{rem}\label{forrem}
Setting $A=0$ and $B=1$, we obtain the same radial limits as Folsom, Ono, and Rhoades in \cite{for}, Theorem 1.2. However, the modular form we subtracted on the left hand side is different (see also Remark \ref{zu}). 
\end{rem}

Next, we treat the case where $g_3$ can be computed directly as an absolutely convergent geometric series. For $x\in\R$, let $\{ x\}$ denote the unique number in $[0,1)$ with $x-\{x\}\in \Z$. 

\begin{thm}\label{conv}
If $\frac{h}{k}\notin \mathcal{Q}_{a,b,A,B}$ and $\{k' \left(\frac{a}{b}+\frac{Ah}{k}\right)\}\in (\tfrac{1}{6},\tfrac{5}{6})$, then
\begin{equation*}
\lim_{q\rightarrow \zeta_h ^k} g_3 (\zeta_b ^a q^A, q^B)=\frac{1}{1-\frac{1}{2-2\cos (2\pi k'\left(\frac{a}{b}+\frac{Ah}{k}\right) )}}\sum_{j=1}^{k'}\frac{\zeta_k ^{hBj(j-1)}}{(\zeta_b ^a \zeta_k ^{hA} , \zeta_b ^{-a} \zeta_k ^{h(B-A)};\zeta_k ^{hB})_{j}}.
\end{equation*}
\end{thm}

Now we need the following beautiful identity of Kang (\cite{kang}, Theorem 1.3.).
\begin{equation}\label{kangsid}
g_3 (x,q) + g_3 (\zeta _3 x , q) + g_3 (\zeta_3 ^2 x, q)= 
\frac{3i(q^3;q^3)_\infty ^3}{(q)_\infty j(x^3, q^3)}.
\end{equation}

If the two shifted values on the left hand side of \eqref{kangsid} converge, we can subtract the modular form on right hand side to obtain a bounded radial limit. This is only possible if $k'$ is not divisible by $3$.

\begin{thm}\label{kang}
Assume that $\frac{h}{k}\notin \mathcal{Q}_{a,b,A,B}$, $\{k' \left(\frac{a}{b}+\frac{Ah}{k}\right)\}\in [0,\tfrac{1}{6})\cup (\tfrac{5}{6},1)$ and $3\nmid k'$. Then we have
\begin{multline*}
\lim_{q\rightarrow \zeta_h ^k}\left( g_3 (\zeta_b ^a q^A, q^B) - \frac{3i(q^{3B}; q^{3B})_\infty ^3}{(q^B; q^B)_\infty j(\zeta_b ^{3a} q^{3A}, q^{3B})} \right)\\
=-\sum_{\ell =1}^2 \frac{1}{1-\frac{1}{2-2\cos (2\pi k'\left(\frac{a}{b}+\frac{Ah}{k}+\frac{\ell}{3}\right) )}}\sum_{j=1}^{k'}\frac{\zeta_k ^{hBj(j-1)}}{(\zeta_3 ^{\ell} \zeta_b ^a \zeta_k ^{hA} , \zeta_3 ^{2\ell} \zeta_b ^{-a} \zeta_k ^{h(B-A)};\zeta_k ^{hB})_{j}}.
\end{multline*}
\end{thm}

If $3|k'$ and $g_3 (x,q)$ diverges, then also the other two summands in the left hand side of \eqref{kangsid} diverge. 


\begin{rem}\label{case4}
In trying to treat the case $3|k'$ and $\{ k' \left(\tfrac{a}{b}+\tfrac{Ah}{k}\right)\} \notin \{\tfrac{1}{6},\tfrac{5}{6}\}$, we first found a naive formula for the radial limits, which does not seem to be true in general due to convergence issues. However, if we additionally require the assumptions of Theorem \ref{pole} to be true, then numerical computations surprisingly suggest the naive formula and the expression in Theorem \ref{pole} to be equal in this overlap case. This leads us to the following conjecture.

\begin{con}
Let $x,q$ be roots of unity with $(x, q/x ;q)_\infty =0$ such that $q$ is a primitive root of unity of order $3k$ and $x^{3k}$ is not a primitive sixth order root of unity. Then we have
\begin{multline*}
\frac{1}{1-x^{3k}+x^{6k}}\sum_{j=1}^k (-1)^j x^{3j-2} q^{-\frac{(3j+1)j}{2}} \left(q (1+x^{3k} q^k) + x(1+x^{3k}q^{2k}) \right)\\
=-x^{-1}+x^{-2}\sum_{j=1}^{3k} (q/x; q)_{j-1} (x;q)_j q^j.
\end{multline*}
\end{con}

We challenge the interested reader to find a proof for this identity.
\end{rem}

\begin{exmp}
Up to a prefactor, all of Ramanujan's fifth order mock theta functions studied in \cite{bklmr} can be expressed as the sum of a specialization of $g_3$ with $\tfrac{a}{b}\in\tfrac{1}{5}\Z$ and $B=5$ or $10$, and a modular form. In these cases, we have $k' \left(\frac{a}{b}+\frac{Ah}{k}\right)=\frac{Ah}{(k,B)}\in \frac{1}{10}\Z$, so at least one of the above Theorems will apply. However, our expressions for the radial limits are different from those in \cite{bklmr}. Comparing the results and using the uniqueness of the radial limits yields many curious corollaries. We give one example of such an identity. The fifth order mock theta function $f_0 (q)$ can be written as $-2q^2 g_3 (q^2 ,q^{10})$ plus a modular form. Here we have $(a,b,A,B)=(0,1,2,10)$ with $\mathcal{Q}_{0,1,2,10}:=\left\lbrace \frac{h}{k}\in\Q : (k,10)|2h   \right\rbrace$. Let now $k$ be a positive even integer with $5\nmid k$, so that $(k,10)=2|2h$ for every $h$. Then we have $\tfrac{h}{k}\in\mathcal{Q}_{0,1,2,10}$ and we can apply Theorem \ref{pole}. On the other hand, for even $k$, we can also apply the first equation of \cite{bklmr}, Theorem 1.1.~(a) to compute the radial limit. Therefore we obtain the following corollary.

\begin{cor}
Let $k\in \N$ with $(k,10)=2$ and $\zeta$ a root of unity of order $k$. Then we have
\begin{equation*}
2-2\zeta^{-2}\sum_{n=1}^{k/2}(\zeta ^8 ; \zeta ^{10})_{n-1} (\zeta^2; \zeta ^{10})_n \zeta^{10n}= -2\sum_{n=0}^{k-1} \zeta ^{(n+1)(n+2)/2}(-\zeta ; \zeta)_n.
\end{equation*}.
\end{cor} 
We obtain similar equations for the other cases of \cite{bklmr}, Theorems 1.1.~and 1.2. 
It would be interesting to find a direct explanation for all of these identities. 
\end{exmp}

\section{Definitions and Useful Formulas}

First we fix some notation. For a non-zero complex number $q$ in the open unit disc, $n\in \N \cup \{\infty \}$ and $a\in \C$, the $q$-Pochhammer symbol is given by
\begin{equation*}
(a)_n := (a ; q)_n :=\prod_{j=0}^{n-1} (1-aq^{j})
\end{equation*}
and we abbreviate
\begin{equation*}
(a_1 ,\dots ,a_m ;q )_n:= \prod_{j=1}^{m}(a_j ;q)_n .
\end{equation*}
We further define the Jacobi triple product as
\begin{equation*}
j(x,q) := (x,q/x,q ; q)_\infty.
\end{equation*}
Now we show that the infinite products we subtracted in Theorems \ref{pole} and \ref{kang} are indeed modular forms.

\begin{lem}
For every $n\in\N$, $\zeta$ a root of unity, and $\alpha\in\Q$, the functions $(q^n)_\infty$ and $j(\zeta q^\alpha , q^n)$ are modular forms in the sense of Remark \ref{modu}. In particular, the infinite products on the left hand-sides of Theorems \ref{pole} and \ref{kang} are modular forms in this sense.
\end{lem}

\begin{proof}
First we have 
$$
(q)_\infty= q^{-1/24}\eta(q),
$$
where $\eta$ is the Dedekind eta-function, which is a modular form of weight $1/2$. Furthermore, the well-known Jacobi triple product identity states that 
$$
j(x,q)=-x^{1/2}q^{-1/8}\vartheta(x,q),
$$
where $\vartheta$ is the Jacobi theta-function, which is a Jacobi form of weight and index $1/2$. This implies that $\vartheta(\zeta q^\alpha, q)$ is a modular form of weight $1/2$ for every root of unity $\zeta$ and $\alpha\in\Q$. Since replacing $q$ by $q^n$ only changes the congruence subgroup, we obtain the statement. 
\end{proof}

We next define the Appell-Lerch sum
\begin{equation*}
m(x,q,z):=\frac{1}{j(z,q)}\sum_{n\in\Z}\frac{(-z)^n q^{\frac{n(n-1)}{2}}}{1-xzq^{n-1}}.
\end{equation*}
We have the following shifting property in the third argument (see \cite{zwe}, Propostion 1.4, (7), using a different notation, or \cite{mor}, Propostion 1.1, (1.4d)):
\begin{equation}\label{shift}
m(x,q,z) - m(x,q,w) = \frac{w (q)_\infty ^3 j(z/w,q)j(xzw, q)}{j(z,q)j(w,q)j(xz,q)j(xw,q)}.
\end{equation}
Moreover, we set
\begin{equation*}
\widetilde{g}(x,q) :=-x\sum_{n\geq 0} \frac{q^{n^2}}{(x)_{n+1} (q/x)_{n}}.
\end{equation*}

The tail sum for $\widetilde{g}$ is defined as
\begin{equation*}
\widetilde{g}_t (x,q) := -x\sum_{n < 0} \frac{q^{n^2}}{(x)_{n+1} (q/x)_{n}}.
\end{equation*}
We want to express $\widetilde{g}_t$ as a sum over positive integers. For this, we use the well-known identity
\begin{equation*}
(a)_{-n} = -\frac{q^{\frac{n(n+1)}{2}}}{a^n (q/a)_n}
\end{equation*} 
for $a\in\C$ and $n\in\N$, to write
\begin{equation}\label{l1}
\widetilde{g}_t (x,q) = \sum_{n\geq 1} (q/x)_{n-1} (x)_n q^n.
\end{equation}

\begin{rem}
Note that $R(x,q)=(1-x)(1+xg_3 (x,q))$ is Dyson's rank generating function and $U(x,q)=(1-x)^{-1}\widetilde{g}_t(x,q)$ is the generating function for strongly unimodal sequences. These are exactly the functions occurring in \cite{for}, Theorem 1.2. 
\end{rem}

Now we need the following relation between $g_3$ and $\widetilde{g}$, which can be found in \cite{bfr}, Theorem 3.1.
\begin{equation}\label{l2}
g_3 (x,q)=-x^{-1} \left( 1+x^{-1}\widetilde{g}(x,q)\right).
\end{equation}

We also require the following identity from the Lost Notebook, which has already been used in \cite{for} to prove their related Theorem 1.2.

\begin{prop}[see \cite{lost}, Entry 3.4.7 or \cite{mor}, Proposition 1.3]
For $a,b\in\C\backslash\{0\}$, we have
\begin{equation*}
\sum_{n=0}^\infty \frac{a^{-n-1}b^{-n}}{(-1/a)_{n+1}(-q/b)_n}q^{n^2}+\sum_{n=1}^\infty (-aq)_{n-1}(-b)_n q^n = \frac{(-aq)_\infty j(-b,q)}{b(q, -q/b; q)_\infty}m(a/b,q,-b).
\end{equation*}
\end{prop}

Plugging in $a=-x^{-1}$ and $b=-x$, we obtain a nice bilateral series identity for $\widetilde{g}$:

\begin{equation}\label{p1}
\widetilde{g}(x,q)+\widetilde{g}_t (x,q)=-\frac{j(x,q)}{x(q)_\infty}m(x^{-2},q,x).
\end{equation}

\begin{rem}\label{zu}
Note that \eqref{p1} is equivalent to Equation (3) of \cite{zud}. Zudilin asked if one can relate the right hand side asymptotically to the Andrews-Garvan crank function $C(x,q)$, in order to obtain a simpler proof of Theorem 1.2.~of \cite{for}. However, in our approach, we end up with a modular form different from $C(x,q)$ in Theorem \ref{pole}. 
\end{rem}

We will also need the following functional equation for $g_3$ (cf.~\cite{gmi}, eqs.~(4.7) and (6.2)):
\begin{equation}\label{feq}
g_3 (xq, q) = -x^3 g_3 (x,q) - x^2 -x.
\end{equation}

\begin{rem}
One may also deduce \eqref{p1} by applying the heuristic presented in \cite{mor} to $\widetilde{g}$ and a functional equation corresponding to \eqref{feq}. 
\end{rem}

With the previous results we can express $g_3$ as a sum of $\widetilde{g}_t$, an Appell-Lerch sum and a Jacobi form. 

\begin{lem}\label{tailid}
For every $z\in \C$ with $j(z,q)\neq 0$, we have
\begin{multline*}
g_3 (x,q)=-x^{-1}+x^{-2}\sum_{n\geq 1}(q/x)_{n-1}(x)_n q^n +\frac{j(x,q)}{x^3 (q)_\infty j(z,q)}\sum_{n\in\Z}\frac{(-z)^n q^{\frac{n(n-1)}{2}}}{1-x^{-2}zq^{n-1}}\\
+\frac{z(q)_\infty j(xz^{-1},q)j(x^{-1}z,q)}{x^3 j(z,q)j(x^{-1},q)j(x^{-2}z,q)}.
\end{multline*}
\end{lem}

\begin{proof}
Equations \eqref{l2} and \eqref{p1} imply
\begin{equation*}
g_3 (x,q) = -x^{-1}+x^{-2}\left( \widetilde{g}_t (x,q) + \frac{j(x,q)}{x(q)_\infty}m(x^{-2},q,x) \right).
\end{equation*}
Now the statement follows by \eqref{l1} and \eqref{shift}.
\end{proof}

If some summands of $g_3$ have a pole, then $\widetilde{g}_t$ becomes a finite sum. The idea is to choose an appropriate $z$, such that the Apell-Lerch sum goes to zero in this case. Now setting $z=x\sqrt{q}$ in Proposition \ref{tailid} yields
\begin{multline}\label{tailid2}
g_3 (x,q)=-x^{-1}+x^{-2}\sum_{n\geq 1}(q/x)_{n-1}(x)_n q^n 
+\frac{j(x,q)}{x^3 (q)_\infty j(x\sqrt{q},q)}\sum_{n\in\Z}\frac{(-x)^n q^{\frac{n^2}{2}}}{1-x^{-1} q^{n-\frac{1}{2}}}\\+\frac{(q)_\infty ^2 j(\sqrt{q},q)^2}{x j(x\sqrt{q},q)^2 j(x,q)}.
\end{multline}
Since we always have $j\left(\pm x\sqrt{q}, q\right)\neq 0$, the choice of the square root $\sqrt{q}$ does not matter.

\section{Case-by-Case Analysis}

\subsection{Poles in the denominators}
First we look at poles of $g_3$ arising from zeros in the denominators.
\begin{lem}
The poles in $g_3 (\zeta _b ^a q^A, q^B)$ arising from zeros in the denominators are exactly at the cusps $\frac{h}{k}\in \mathcal{Q}_{a,b,A,B}$.
\end{lem}

\begin{proof}
The statement follows from \cite{br}, Proposition 3.2, since the summands of the function $g_2$ examined in \cite{br} have exactly the same denominators as the summands of $g_3$. Moreover, the numerators of the summands of $g_3$ never vanish.  
\end{proof}

\begin{prop}\label{lim0}
For every $\frac{h}{k}\in \mathcal{Q}_{a,b,A,B}$, we have
\begin{equation*}
\lim_{q\rightarrow \zeta_h ^k}\frac{j(\zeta _b ^a q^A,q^B)}{(q^B ; q^B)_\infty j(\zeta _b ^a q^{A+\frac{B}{2}},q^B)}\sum_{n\in\Z}\frac{(-1)^n\zeta_b ^{an} q^{\frac{Bn^2}{2}+An}}{1-\zeta_b ^{-a} q^{(B-A)n-\frac{B}{2}}}=0.
\end{equation*}
\end{prop}

The proof is completely analogous to the proof of Lemma 4.2 in \cite{br}, since $j(\zeta _b ^a \zeta_k ^{h\left(A+\frac{B}{2}\right)},\zeta_k ^{hB})$ never vanishes if $\frac{h}{k}\in \mathcal{Q}_{a,b,A,B}$.

\begin{proof}[Proof of Theorem \ref{pole}]
After using Equation \eqref{tailid2} and Proposition \ref{lim0}, it remains to show that 
\begin{equation*}
(\zeta_b ^{-a} \zeta_k ^{h(B-A)};\zeta_k ^{hB})_{n-1} (\zeta_b ^a \zeta_k ^{hA}; \zeta_k ^{hB})_{n}
\end{equation*} 
vanishes for $n>k'$. We have $xq^n=1$ for $n\in \N$ if and only if 
\begin{equation*}
\tfrac{ak}{b}+hA+hBn \equiv 0 \pmod k. 
\end{equation*}
For $\frac{h}{k}\in\mathcal{Q}_{a,b,A,B}$ this equation has a unique solution $\pmod{k'}$. Thus for some 
$n_0\in\{ 1,\dots ,k'\}$ we have $xq^{n_0}=1$ and therefore $(x)_n=0$ for every $n> n_0$.
\end{proof}

\subsection{Convergent geometric series}
Now we examine the case when the radial limit of $g_3$ exists as a convergent sum. 

\begin{proof}[Proof of Theorem \ref{conv}]
We show that the sum $g_3 (\zeta_b ^a \zeta_k ^{hA}, \zeta_k ^{hB} )$ converges absolutely in the above case by rearranging terms 
\begin{multline*}
g_3 (\zeta_b ^a \zeta_k ^{hA} , \zeta_k ^{hB} ) 
= \sum_{n\geq 1}\frac{\zeta_k ^{hBn(n-1)}}{(\zeta_b ^a \zeta_k ^{hA} , \zeta_b ^{-a} \zeta_k ^{h(B-A)} ; \zeta_k ^{hB})_{n}}
= \sum_{m\geq 0} \sum_{j=1}^{k'}\frac{\zeta_k ^{hB(mk'+j)(mk'+j-1)}}{(\zeta_b ^a \zeta_k ^{hA} , \zeta_b ^{-a} \zeta_k ^{h(B-A)} ; \zeta_k ^{hB})_{mk'+j}}\\
= \sum_{m\geq 0} \sum_{j=1}^{k'}\frac{\zeta_k ^{hBj(j-1)}}{(\zeta_b ^{a}\zeta_k ^{hA} , \zeta_b ^{-a} \zeta_k ^{h(B-A)} ; \zeta_k ^{hB})_{k'}^m(\zeta_b ^{a}\zeta_k ^{hA} , \zeta_b ^{-a} \zeta_k ^{h(B-A)} ; \zeta_k ^{hB})_{j}} \\
= \sum_{j=1}^{k'}\frac{\zeta_k ^{hBj(j-1)}}{ (\zeta_b ^a \zeta_k ^{hA}, \zeta_b ^{-a} \zeta_k ^{h(B-A)} ; \zeta_k ^{hB})_{j}}\sum_{m\geq 0} \frac{1}{(\zeta_b ^a \zeta_k ^{hA},\zeta_b ^{-a} \zeta_k ^{h(B-A)} ; \zeta_k ^{hB} )_{k'} ^m }.
\end{multline*}
The second sum is a geometric series that converges absolutely if and only if 
\begin{equation*}
|(\zeta_b ^a \zeta_k ^{hA},\zeta_b ^{-a} \zeta_k ^{h(B-A)} ; \zeta_k ^{hB} )_{k'}|>1.
\end{equation*}
Now consider the expression $(x,q/x;q)_{k'}$ for $q$ a primitive $k'$-th root of unity.
We use the factorization
\begin{equation*}
1-x^{k'}=\prod_{j=0}^{k'-1}(1-xq^j)
\end{equation*}
to obtain 
\begin{multline*}
(x,q/x;q)_{k'} = \prod_{j=0}^{k'-1}(1-xq^{j})(1-x^{-1}q^{j+1})=\prod_{j=0}^{k'-1}(1-xq^j)\prod_{j=0}^{k'-1}(1-x^{-1}q^j)\\
=(1-x^{k'})(1-x^{-k'})=2-2\operatorname{Re}(x^{k'}).
\end{multline*}
Setting $x=\zeta_b ^a \zeta_k ^{hA}$, we get
\begin{equation*}
\operatorname{Re}(x^{k'})=\cos (2\pi k' \left(\tfrac{a}{b}+\tfrac{Ah}{k}\right) ) 
\end{equation*}
and thus
\begin{equation*}
(\zeta_b ^a \zeta_k ^{hA},\zeta_b ^{-a} \zeta_k ^{h(B-A)} ; \zeta_k ^{hB} )_{k'} = 2-2\cos (2\pi k' \left(\tfrac{a}{b}+\tfrac{Ah}{k}\right)). 
\end{equation*}
Therefore the series converges absolutely if and only if
\begin{equation*}
\cos (2\pi k' \left(\tfrac{a}{b}+\tfrac{Ah}{k}\right)) < \tfrac{1}{2},
\end{equation*}
or equivalently
\begin{equation*}
\{ k' \left(\tfrac{a}{b}+\tfrac{Ah}{k}\right)\} \in (\tfrac{1}{6},\tfrac{5}{6}).
\end{equation*}
By Abel's Theorem, the limit of the series equals the radial limit of $g_3$ from within the unit disc.
\end{proof}

\subsection{Shifting by third roots of unity}
\begin{proof}[Proof of Theorem \ref{kang}]
We use Kang's identity \eqref{kangsid}. Note that under the assumptions of Theorem \ref{kang}, if $\frac{a}{b}\in\{\tfrac13, \tfrac23\}$, we always have $A\neq 0$ and $B\nmid A$, so that all three summands of \eqref{kangsid} are well-defined for $|x|,|q| <1$. 
First we show that $\frac{h}{k}\notin \mathcal{Q}_{3a+b,3b,A,B}\cup \mathcal{Q}_{3a+2b,3b,A,B}$, so that the summands of $g_3 (\zeta_3 \zeta_b ^a \zeta_k ^{hA}, \zeta_k ^{hB})$ and $g_3 (\zeta_3 ^2 \zeta_b ^a \zeta_k ^{hA}, \zeta_k ^{hB})$ have non-vanishing denominators. Here we abuse notation and write $\mathcal{Q}_{a,b,A,B}$ for $\mathcal{Q}_{\frac{a}{(a,b)},\frac{b}{(a,b)},A,B}$ if $(a,b)>1$. Suppose that
\begin{equation*}
(B,k) | \tfrac{ak}{b} + Ah + \tfrac{\ell k}{3} 
\end{equation*}
for some $\ell\in \{ 1,2\}$.
This means that $k' (\tfrac{a}{b} + \tfrac{Ah}{k}+ \tfrac{\ell}{3})$ is an integer. Since $k'$ is not divisible by $3$, it follows that $\{k' \left(\frac{a}{b}+\frac{Ah}{k}\right)\}\in \{\tfrac{1}{3} , \tfrac{2}{3}\} $, which contradicts our assumption. \\
Since we also have 
\begin{equation*}
\{k' \left(\tfrac{a}{b}+\tfrac{Ah}{k}+\tfrac{1}{3}\right)\}, \{k' \left(\tfrac{a}{b}+\tfrac{Ah}{k}+\tfrac{2}{3}\right)\} \in (\tfrac{1}{6}, \tfrac{5}{6}), 
\end{equation*}
we see by Theorem \ref{conv} that $g_3 (\zeta_3 \zeta_b ^a \zeta_k ^{hA}, \zeta_k ^{hB})$ and $g_3 (\zeta_3 ^2 \zeta_b ^a \zeta_k ^{hA}, \zeta_k ^{hB})$ are absolutely convergent sums. Together with \eqref{kangsid} the statement follows.
\end{proof}

\subsection{Future directions and discussion}
Apart from the case where $g_3(x,q)$ diverges and $k'$ is divisible by $3$, we are left with the case where $\frac{h}{k}\notin \mathcal{Q}_{a,b,A,B}$ and $\{ k' \left(\tfrac{a}{b}+\tfrac{Ah}{k}\right)\} \in \{\tfrac{1}{6},\tfrac{5}{6}\}$. 
Again we set $x=\zeta_b ^a \zeta_k ^{hA}$ and $q=\zeta_k ^{hB}$.
Then $x$ is a $6k'$-th root of unity of order $\pm 1\pmod{6}$, say $x=\zeta_{6k'} ^{\pm1+6\ell}=\zeta_{6k'}^{\pm1}\zeta_{k'}^{\ell}$, $\ell\in\Z$. Notice that since $(k',hB')=1$, we can always find an integer $n$ satisfying $hB'n\equiv -\ell \pmod{k'}$, meaning that
\begin{equation*}
q^n=\zeta_{k'}^{hB'n}=\zeta_{k'}^{-\ell}.
\end{equation*}
By iteratively applying \eqref{feq}, we can reduce this to the cases $x=\zeta_{6k'}^{\pm1}$. Also, by easy computation, we can obtain
\begin{equation}\label{inv}
g_3(x^{-1},q)=g_3(xq,q)=-x^3 g_3(x,q)-x^2 -x.
\end{equation} 
Because of this identity, it is enough to consider the case $x=\zeta_{6k'}$. Now we obtain
\begin{equation*}
q=\zeta_{k'} ^{hB'} =x^{6hB'}.
\end{equation*}

Assume that $hB'\equiv 1 \pmod{k'}$, i.e.~$q=\zeta_{k'}=x^6.$ In this case, we can use the following mock theta ``conjecture'' (see \cite{gmi}, eq.~(7.3)):
\begin{equation*}
g_3(x,x^6) = -\frac{1}{2x}+\frac{x}{2}g_3 (x^3, x^6) + \frac{(x^2;x^2)_\infty ^4}{2x (x;x)_\infty ^2 (x^6; x^6)_\infty }.
\end{equation*}
First, we can apply Theorem \ref{conv} to $g_3(x^3,x^6)$, because $g_3 (x^3, x^6)=g_3 (\zeta_{2k'}, \zeta_{k'})$ and $k'\cdot\frac{1}{2k'}=\frac{1}{2}\in\left(\frac{1}{6},\frac{5}{6}\right)$. Thus we have
\begin{equation*}
 g_3 (\zeta_{2k'}, \zeta_{k'})=\frac{4}{3}\sum_{j=1}^{k'}\frac{\zeta_{k'} ^{j(j-1)}}{(\zeta_{2k'} ; \zeta_{k'})_{j}^{2}}.
\end{equation*}
Next, we look at the remaining modular term. It can be easily checked that
\begin{equation*}
\frac{(x^2;x^2)_\infty ^4}{(x;x)_\infty ^2 (x^6; x^6)_\infty }=(-x;x)_\infty^4(x;x)_\infty(x,x^2,x^3,x^4,x^5;x^6)_\infty,
\end{equation*} 
which vanishes for $x=\zeta_{6k'}$.\\

Combining all of the above, we obtain
\begin{equation*}
\lim_{q\rightarrow \zeta_{k}^h} g_3 (\zeta_b ^a q^A, q^B)=-\frac{1}{2\zeta_{6k'}}+  \frac{2\zeta_{6k'}}{3}\sum_{j=1}^{k'}\frac{\zeta_{k'} ^{j(j-1)}}{(\zeta_{2k'} ; \zeta_{k'})_{j}^{2}}
\end{equation*}
if $\frac{h}{k}\notin \mathcal{Q}_{a,b,A,B}$, $ k' \left(\tfrac{a}{b}+\tfrac{Ah}{k}\right) =\tfrac{1}{6}$, and $hB'\equiv 1\pmod{k'}$.\\

If $hB'\not\equiv 1\pmod{k'}$, one can use the relation \cite{gmi}, eq.~(6.1) between $g_2$ and $g_3$ and the results for $g_2$ in \cite{br} to compute the modular forms to be subtracted and the radial limits of $g_3$. Unfortunately, this method is quite laborious and not straightforward. It would be interesting to find direct approach to these cases.

\end{document}